\documentclass[a4paper,oneside,12pt]{article}
\usepackage{amsmath,amsfonts,amscd,amssymb,amsthm}
\usepackage{longtable,geometry}
\usepackage[english]{babel}
\usepackage[utf8]{inputenc}
\usepackage[active]{srcltx}
\usepackage[T1]{fontenc}
\usepackage{graphicx}
\usepackage{pstricks}
\usepackage{bbm}
\usepackage{MnSymbol}
\usepackage{nicefrac}
\usepackage{calrsfs}
\usepackage[pdfborder={0 0 0},colorlinks=true,linkcolor=black,citecolor=black,urlcolor=blue]{hyperref}

\newtheorem{theorem}{Theorem}

\newtheorem{lemma}[theorem]{Lemma}
\newtheorem{proposition}[theorem]{Proposition}
\newtheorem*{claim*}{Claim}

\theoremstyle{remark}

\newtheorem{problem}[theorem]{Problem}

\newtheorem{remark}[theorem]{Remark}
\newtheorem{conjecture}[theorem]{Conjecture}

\renewenvironment{proof}[1][\relax]%
  {\pushQED{\qed}
   \paragraph{Proof\ifx#1\relax.\else~of #1\fi}}%
  {~\hfill\popQED\par\medskip}

\renewcommand{\O}{\Omega}

\renewcommand{\d}{\delta}

\newcommand{\Z}{\mathbb{Z}}

\newcommand{\g}{\gamma}

\newcommand{\DD}{\mathbb{D}}

\newcommand{\dist}{\operatorname{dist}}
\DeclareMathOperator{\bdist}{box dist}

\newcommand{\SLE}{\textrm{SLE}}

\newcommand{\N}{\mathbb{N}}

\renewcommand{\P}{\mathbb{P}}

\newcommand{\hrefnlu}[2]{\href{#1}{\nolinkurl{#2}}}

\title{Supercritical self-avoiding walks are space-filling}
\author{Hugo Duminil-Copin, Gady Kozma and Ariel Yadin}
\begin{document}

\date{}

\maketitle

\section{Introduction}

In 1953, Paul Flory \cite{Flory} proposed considering self-avoiding walks (\emph{i.e.}\ visiting every vertex at most once) on a lattice as a model for polymer chains. Self-avoiding walks have turned out to be a very interesting object, leading to rich mathematical theories and challenging questions; see \cite{MadrasSlade,hugo}.

Denote by $c_n$ the number of $n$-step self-avoiding walks on the hypercubic lattice ($\mathbb Z^d$ with edges between nearest neighbors) started from some fixed vertex, \emph{e.g.}\ the origin. Elementary bounds on $c_n$ (for instance $d^n\leq c_n\leq 2d(2d-1)^{n-1}$) guarantee that $c_n$ grows exponentially fast. Since an $(n+m)$-step self-avoiding walk can be uniquely cut into an $n$-step self-avoiding walk and a parallel translation of an $m$-step self-avoiding  walk, we infer that $c_{n+m}\leq c_nc_m$,
from which it follows that there exists $\mu=\mu(d)\in(0,+\infty)$ such that $\mu:=\lim_{n\rightarrow \infty}c_n^{1/n}$. The positive real number $\mu$ is called the \emph{connective constant} of the lattice. The connective constant can be approximated in a number of ways, yet no closed formula exists in general. In the case of the hexagonal lattice, it was recently proved to be equal to $\sqrt{2+\sqrt 2}$ in \cite{DS10}.

As it stands, the model has a strong combinatorial flavor. A more
geometric variation was suggested by Lawler, Schramm and Werner
\cite{LSW5}. Let us describe their construction now --- they were
interested in the two-dimensional case, but here we will not make this restriction.
Let $\O$ be a simply connected domain in $\mathbb R^d$ with two points
$a,b$ on the boundary. For $\d>0$, let $\O_\d$ be the largest
connected component of $\O \cap \d\mathbb Z^d$ and let $a_\d$, $b_\d$ be the two sites of $\O_\d$ closest to $a$ and $b$ respectively. 
We think of $(\O_\d,a_\d,b_\d)$ as being an approximation of
$(\O,a,b)$. See figure \ref{fig:D_delta}.

\begin{figure}
\begin{center}
\includegraphics[width=0.50\textwidth]{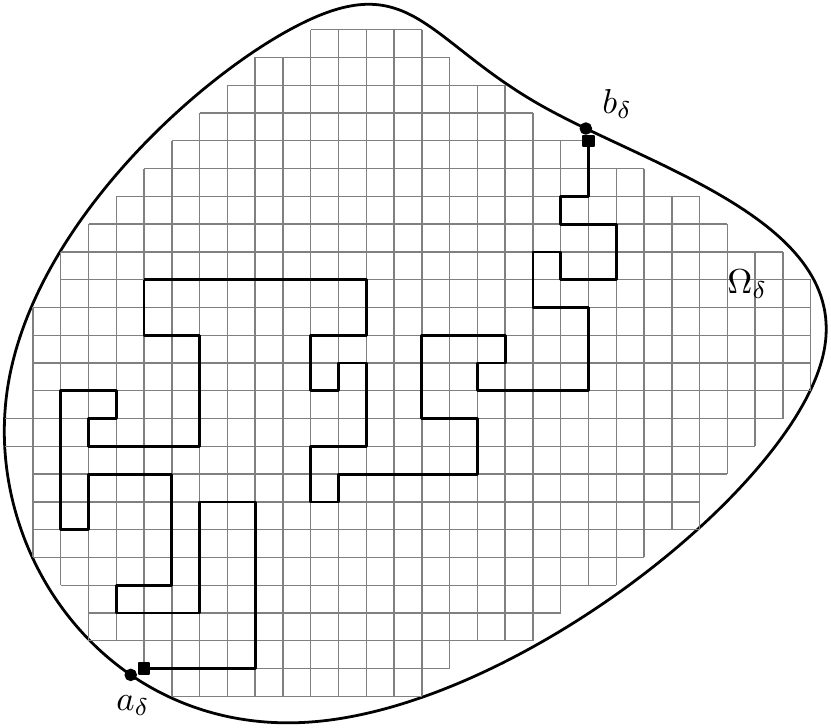}
\end{center}
\caption{\label{fig:D_delta}A domain $\O$ with two points $a$ and $b$
  on the boundary (circles) and the graph $\Omega_\d$. The points $a_\d$
  and $b_\d$ are depicted by squares. An example of a possible walk from $a_\d$ to $b_\d$ is presented. Note that there is a finite number of them.}
\end{figure}

Let $x>0$. On $(\O_\d,a_\d,b_\d)$, define a probability measure on the finite set of self-avoiding walks in $\O_\d$ from $a_\d$ to $b_\d$ by the formula
\begin{equation} \mathbb P_{(\O_\d,a_\d,b_\d,x)}(\gamma)~=~\frac{x^{|\g|}}{Z_{(\O_\d,a_\d,b_\d)}(x)}\end{equation}
where $|\g|$ is the length of $\g$ ({\em i.e.}\ the number of edges), and $Z_{(\O_\d,a_\d,b_\d)}(x)$ is a normalizing factor. A random curve $\g_\d$ with law $\P_{(\O_\d,a_\d,b_\d,x)}$ is called the \emph{self-avoiding walk} with parameter $x$ in $(\Omega_\d,a_\d,b_\d)$.
The sum $Z_{(\O_\d,a_\d,b_\d)}(x)=\sum_{\g}x^{|\g|}$ (with the sum taken over all self-avoiding walks in $\O_\d$ from $a_\d$ to $b_\d$) is sometimes called the
{\em partition function} (or {\em generating function}) of self-avoiding walks from $a_\d$ to $b_\d$ in the domain $\O_\d$.

When the domain $(\Omega,a,b)$ is fixed, we are interested in the scaling limit of the family $(\g_\d)$, {\em i.e.}\ its geometric behavior when $\d$ goes to $0$. The qualitative behavior is expected to differ drastically depending on the value of $x$. A phase transition occurs at the value $x_c=1/\mu$, where $\mu$ is the connective constant:

\paragraph{When $x<1/\mu$:} $\g_\d$ converges to a
deterministic curve corresponding to the geodesic between $a$ and $b$
in $\O$ (assuming it is unique --- otherwise some adaptations need to
be done). When rescaled, $\g_\d$ should have Gaussian fluctuation of
order $\sqrt \delta$ around the geodesic. The strong results of Ioffe \cite{Ioffe} on the unrestricted self-avoiding
walk would be a central tool for proving such a statement, though we are not aware of a reference for the details.

\paragraph{When $x=1/\mu$:} $\g_\d$ should converge to a random simple
curve. In dimensions four and higher, the limit is believed to be
a Brownian excursion from $a$ to $b$ in the domain $\Omega$. This is heuristically related to a number of
rigorously proved results: in dimensions five and above to the work of
Brydges and Spencer \cite{BS85} and Hara and Slade
\cite{HaraSlade1,HaraSlade2} who showed that unrestricted
self-avoiding walk converges to Brownian motion (see also the book
\cite{MadrasSlade}). 
Dimension four, the so-called {\em upper critical dimension}, is much
harder, but recently some impressive results have been achieved using
a supersymmetric renormalization group approach. These results are
limited to continuous time weakly self-avoiding walk, see
\cite{BrydgesImbrieSlade,BS10,BDS11} and references within. 

In dimension two, the scaling limit is conjectured to be the Schramm-L\"owner Evolution of parameter $\nicefrac{8}{3}$, and
in fact it was pointed out that this is true {\em if} the scaling limit exists as a continuous curve and is
conformally invariant \cite{LSW5}. 

Finally, dimension three remains a mystery, and there is no clear candidate for the scaling limit of self-avoiding walk.

\paragraph{When $x>1/\mu$:} $\g_\d$ is expected to become space-filling in the following sense: for any open set $U\subset \Omega$,
$$\P_{(\O_\d,a_\d,b_\d,x)}[\g_\delta\cap U=\emptyset]\rightarrow 0$$
when $\d$ goes to 0. On the one hand, let us mention that it is not clear in which sense (if any) $(\g_\d)$ has a scaling limit when $d\ge 3$. On the other hand,
the scaling limit is predicted \cite[Conjecture 3]{Sm2} to exist in dimension two (in the case of the hexagonal lattice at least). It should be the Schramm-L\"owner Evolution of parameter 8, which is conformally invariant.

One cannot hope that $\g_\d$ would be space-filling in the strictest
possible sense, namely that every vertex is visited. Nevertheless,
one can quantify the size of the biggest hole not visited by the
walk. The subject of this paper is the proof of a result which
quantifies how $\g_\d$ becomes space filling. Here is a precise
formulation. 
\begin{theorem}\label{space filling improved}
Let $\DD$ be the unit disk and let $a$ and $b$ be two points on its boundary. For every $x>1/\mu$, there exist $\xi=\xi(x)>0$ and $c=c(x)>0$ such that
$$\P_{(\DD_\d,a_\d,b_\d,x)}\big[\text{there exists a component of
  }\DD_\d\setminus \Gamma_\d^\xi\text{ with cardinality }> c\log
  (1/\d)\big]\rightarrow 0$$
when $\d\rightarrow 0$, where $\Gamma_\d^\xi$ is the set of sites in $\DD_\d$ at graph distance less than $\xi$ from $\g_\d$.
\end{theorem}

The theorem is stated for the unit disk to avoid various connectivity
problems. Indeed, assume that at some given scale $\d$ our domain
$\Omega$ has a part which is connected by a ``bridge'' of width
$\d$. Then the graph $\O_\d$ will have a large part connected by a
single edge, which does not leave the self-avoiding walk enough space
to enter and exit. Thus an analog of theorem \ref{space filling improved} will not
hold for this $\Omega$. It is not difficult to construct a single domain
$\O$ with such ``mushrooms'' in many scales. See the figure on the
right. In order to solve this issue,
one can start from an arbitrary domain and expand it
microscopically. This gives rise to the following formulation:
\newdimen\wide
\wide=11.8cm
\newdimen\myht
\myht=4.3cm
\parshape 6 0pt \wide 0pt \wide 0pt \wide 0pt \wide 0pt \wide 0pt \hsize
\vadjust{\kern -\myht \vtop to \myht {\hbox to \hsize{\hss\includegraphics{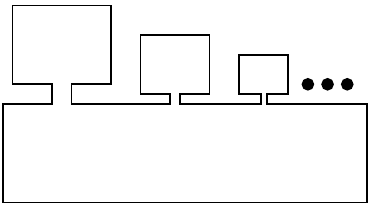}}}}

\begin{theorem}\label{space filling expand}
Let $(\Omega,a,b)$ be a bounded domain with two points on the boundary. For every $x>1/\mu$, there exist $\xi=\xi(x)>0$ and $c=c(x)>0$ such that
$$\P_{((\O+B(\xi\d))_\d,a_\d,b_\d,x)}\big[\exists\text{ a component of
  }(\O+B(\xi\d))_\d\setminus \Gamma_\d^\xi\text{ larger than }c\log
  (1/\d)\big]\rightarrow 0$$
when $\d\rightarrow 0$, where $\Gamma_\d^\xi$ is the set of sites in $(\O+B(\xi\d))_\d$ at graph distance less than $\xi$ from $\g_\d$.
\end{theorem}
Here $\O+B(\xi\d)$ is $\O$ expanded by $\xi\d$ i.e.
\[
\O+B(\xi\d)=\{z:\dist(z,\O)<\xi\d\}.
\]
Since $\xi$ depends only on $x$, and $\d\to 0$, this is a microscopic
expansion.\parshape 0
\bigbreak
The strategy of the proof is fairly natural. We first prove that in
the supercritical phase, one can construct a lot (compared to their energy) of self-avoiding polygons in a prescribed box. Then, we show that the self-avoiding walk cannot leave holes that are too large, since adding polygons in the big holes to the self-avoiding walk would increase the entropy drastically while decreasing the energy in a reasonable way. In particular, a comparison energy/entropy shows that self-avoiding walks leaving big holes are unlikely. We present the proof {\bf only in the case $d=2$}, even though the reasoning carries over to all dimensions without difficulty (see Remark~\ref{generalization}). One can also extend the result to other lattices with sufficient symmetry in a straightforward way (for instance to the hexagonal lattice).

\section{Self-avoiding polygons in a square}

In this section, we think of a walk as being indexed by (discrete)
time $t$ from 0 to $n$. For $m>0$, let $P_m$ be the set of
self-avoiding polygons in $[0,2m+1]^2$ that touch the middle of every
face of the square: more formally, such that the edges
$[(m,0),(m+1,0)]$, $[(2m+1,m),(2m+1,m+1)]$, $[(m,2m+1),(m+1,2m+1)]$
and $[(0,m),(0,m+1)]$ belong to the polygon, see figure \ref{fig:construction_loops}. For $x>0$, let $Z_m(x)$ be the partition function (with parameter $x$) of $P_m$, \emph{i.e.}
$$Z_m(x)~=~\sum_{\g\in P_m}x^{|\g|}.$$

\begin{figure}
\begin{center}
\includegraphics[width=0.60\textwidth]{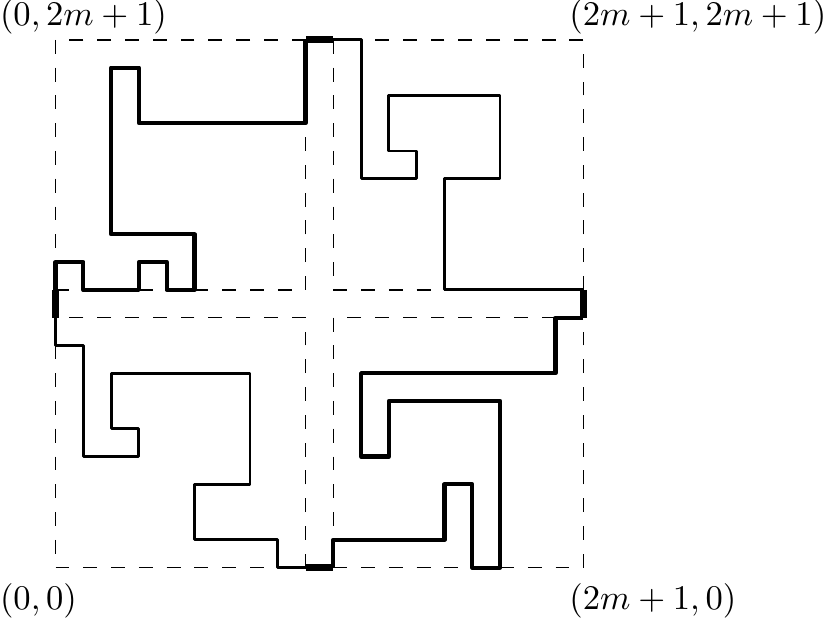}
\end{center}
\caption{\label{fig:construction_loops}by concatenating four walks in squares of size $m$ (plus four edges), one obtains an element of $P_m$, i.e. a loop in the square of size $2m+1$ going through the middle of the sides.}
\end{figure}

\begin{proposition}\label{proposition:loop}
For $x>1/\mu$, we have $\limsup_{m\rightarrow \infty} Z_m(x)~=~\infty$.
\end{proposition}

It is classical that the number of self-avoiding walks with certain
constraints grows at the same exponential rate as the number of
self-avoiding walks without constraints (we will show it in our
context in the proof of Lemma~\ref{lemma:number of rectangular SAWs} below). For instance, let $x(v)$ and $y(v)$ be the first and the second coordinates of the vertex $v$.
The number $b_n$ of {\em self-avoiding bridges} of length $n$, meaning self-avoiding walks $\gamma$ of length $n$ such that $y(\g_0)=\min_{t\in[0,n]} y(\g_t)$ and $y(\g_n)=\max_{t\in[0,n]} y(\g_t)$, satisfies
\begin{equation}\label{eq:42}e^{-c\sqrt n}\mu^n\le b_n\le \mu^n\end{equation}
for every $n$ \cite{HammersleyWelsh} (see also \cite{MadrasSlade} for a modern exposition). This result harnesses the following theorem on integer partitions
which dates back to 1917.

\begin{theorem}[Hardy \& Ramanujan \cite{HR17}]
\label{IntParts}
For an integer $A\geq 1$, let $P_D(A)$ denote the number of ways of
writing $A=A_1+\dotsb+A_k$ with $A_1>\dotsb>A_k\geq 1$, for any $k\geq 1$.
Then
\begin{equation*}
\label{IntPartAsymp}
\log P_D(A) \sim \pi \sqrt{ \frac{A}{3}}
\end{equation*}
as $A\to\infty$.
\end{theorem}

In the following, we need a class of walks with even more restrictive constraints. A {\em squared walk (of span} $k$) is a self-avoiding walk such that $\g_0=(0,0)$, $\g_n=(k,k)$ and $\g\subset[0,k]^2$.

\begin{lemma}
\label{lemma:number of rectangular SAWs}
For $c$ sufficiently large and $n$ even, the number $a_n$ of squared walks of length $n$ satisfies
$$ a_n \geq \mu^n e^{-c \sqrt n}.$$
\end{lemma}

\begin{figure}
\begin{center}
\includegraphics[width=0.50\textwidth]{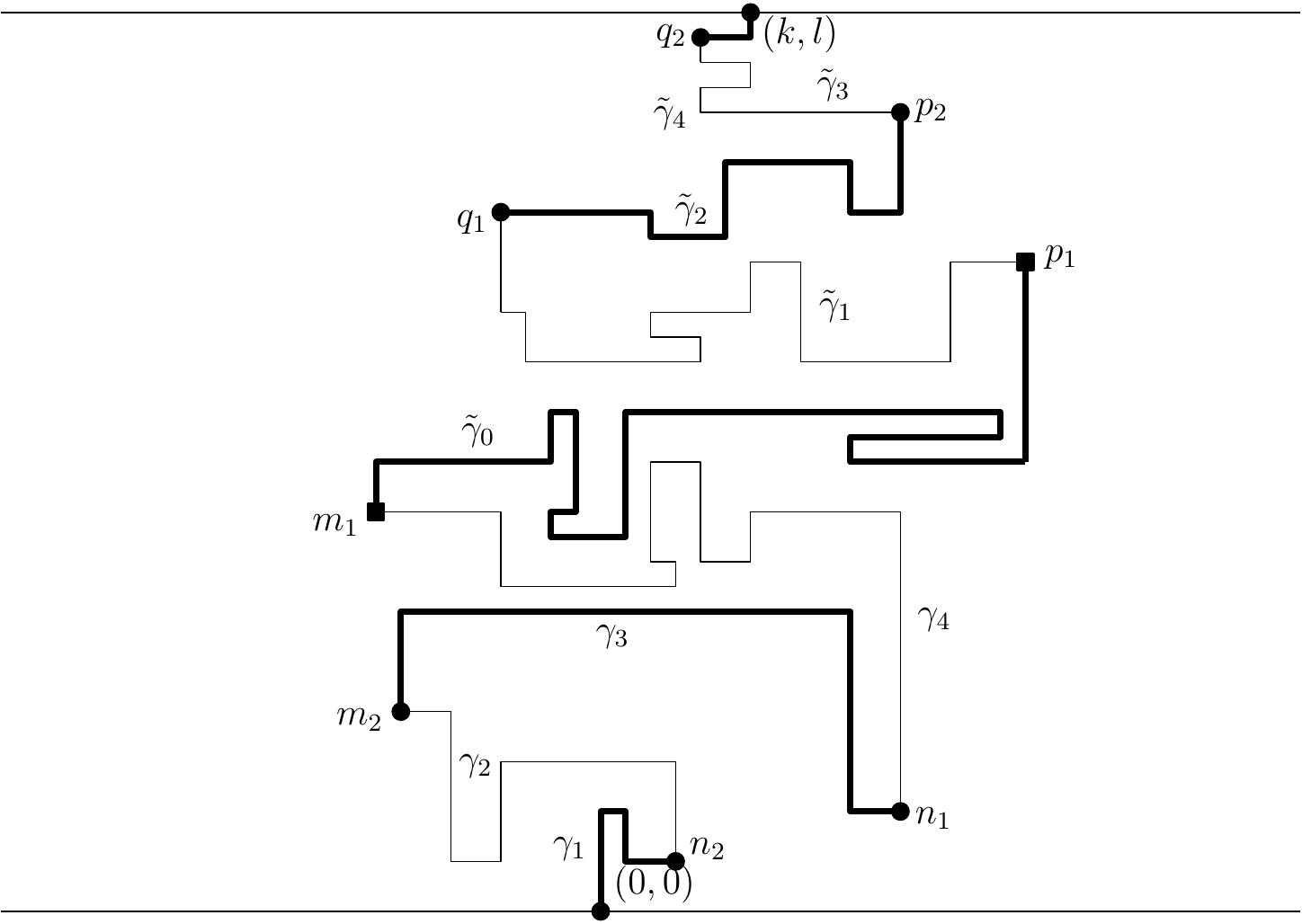}
\end{center}
\caption{\label{fig:bridge_rectangle}The decomposition of a bridge into walks. One can construct a squared walk in a rectangle by reflecting non-bold walks and then concatenating all the walks together.}
\end{figure}

\paragraph{Proof step 1: Rectangles.}
Let $\Lambda_n$ be the set of self-avoiding bridges of length $n$ starting at the origin. Let $\Sigma_n$ be the set of $n$-step self-avoiding walks for which there exists $(k,l)$ such that $\g_0=(0,0)$, $\g_n=(k,l)$ and $\g\subset[0,k]\times[0,l]$. We construct a map from $\Lambda_n$ to $\Sigma_n$. 

Fix $\g\in \Lambda_n$ and denote by $m_1$ the first time at which
$x(\g_{m_1})=\min_{t\in[0,n]} x(\g_t)$, see figure \ref{fig:bridge_rectangle}. Then, define $n_1$ to be the first time at which $x(\g_{n_1})=\max_{t\in[0,m_1]}x(\g_t)$. One can then define recursively $m_k$, $n_k$, by the formul\ae
\begin{align*}m_k&=~\min\{r\leq n_{k-1}~:~x(\g_{r})=\min_{t\in[0,n_{k-1}]}x(\g_t)\}\\
n_k&=~\min\{r\leq m_k~:~x(\g_{r})=\max_{t\in[0,m_k]}x(\g_t)\}\end{align*}
We stop the recursion the first time $m_k$ or $n_k$ equals 0. For
convenience, if the first time is $n_k$, we add a further step
$m_{k+1}=0$. We are then in possession of a sequence of integers
$m_1>n_1>m_2>\dots>m_r\ge n_r\ge0$ and a sequence of walks
$\g_{2r-1}=\g[n_1,m_1]$, $\g_{2r-2}=\g[m_2,n_1],\dots,$
$\g_1=\g[0,m_{r}]$. Note that the width of the walks $\gamma_i$ is
strictly increasing (see figure \ref{fig:bridge_rectangle} again).

Similarly, let $p_1$ be the last time at which $x(\g_{p_1})=\max_{t\in[m_1,n]} x(\g_t)$ and $q_1$ the last time at which $x(\g_{q_1})=\min_{t\in[p_1,n]}x(\g_t)$. Then define recursively $p_k$ and $q_k$ by the following formula
\begin{align*}
p_k&=~\max\{r\geq q_{k-1}~:~x(\g_{r})=\max_{t\in[q_{k-1},n]}x(\g_t)\}\\
q_k&=~\max\{r\leq p_k~:~x(\g_{r})=\min_{t\in[p_k,n]}x(\g_t)\}
\end{align*}
This procedure stops eventually and we obtain another sequence of
walks $\tilde\g_0=\g[m_1,p_1]$, $\tilde\g_1=\g[p_1,q_1]$, etc. This
time, the width of the walks is strictly decreasing, see figure \ref{fig:bridge_rectangle} one more time.

For a walk $\omega$, we set $\sigma(\omega)$ to be its reflexion with respect to the vertical line passing through its starting point. Let $f(\g)$ be the concatenation of $\g_1$, $\sigma(\g_2)$, $\g_3,\dots,\sigma(\g_r)$, $\tilde\g_0$, $\sigma(\tilde\g_1)$, $\tilde\g_2$ and so on. This walk is contained in the rectangle with corners being its endpoints so that $f$ maps $\Lambda_n$ on $\Sigma_n$.

In order to estimate the cardinality of $\Sigma_n$, we remark that each element of $\Sigma_n$ has a limited number of possible preimages under $f$. More precisely, the map which gives $f(\g)$ and the widths of the walks $(\g_i)$ and $(\tilde \g_i)$ is one-to-one (the reverse procedure is easy to identify). The number of possible widths for $\g_i$ and $\tilde \g_i$ is the number of pairs of {\em decreasing} sequences partitioning an integer $l\leq n$. This number is bounded by $e^{c\sqrt n}$ (Theorem~\ref{IntParts}). Therefore, the number of possible preimages under $f$ is bounded by $e^{c\sqrt n}$.
Using \eqref{eq:42}, the cardinality of $\Sigma_n$ is thus larger than $e^{-c\sqrt n}b_n\geq e^{-2c\sqrt n}\mu^n$.

So far $n$ was not restricted to be even.

\begin{figure}[h]
\begin{center}
\includegraphics[width=0.50\textwidth]{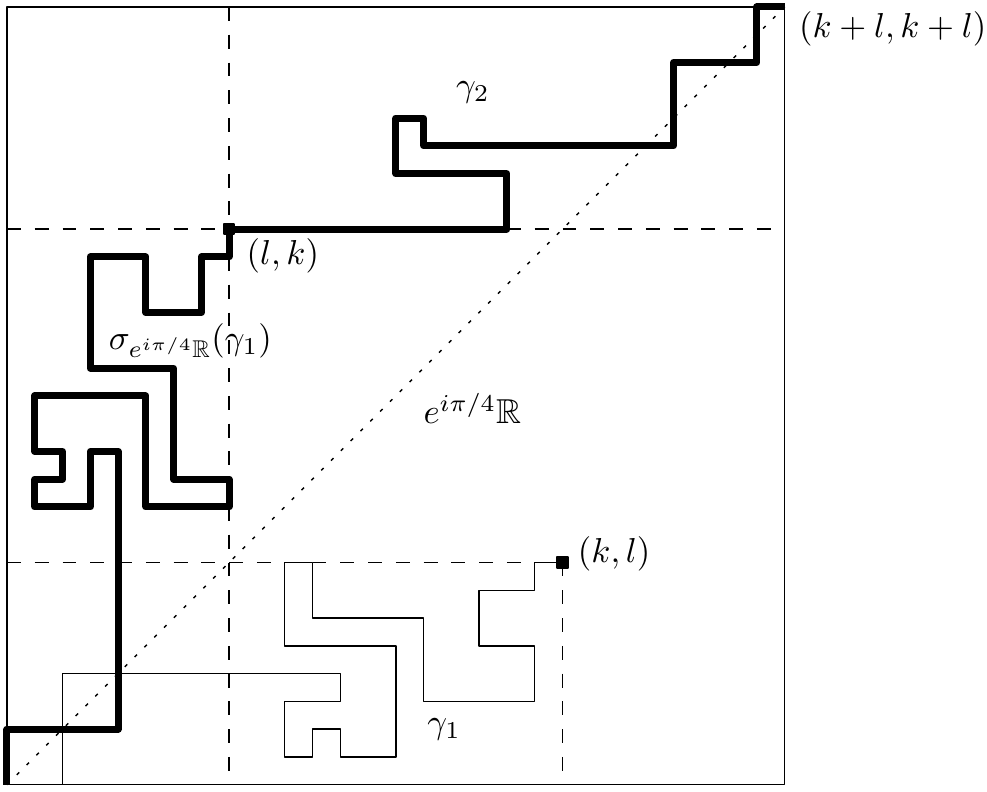}
\end{center}
\caption{\label{fig:rectangle_square}This figure depicts the passage of two walks in the rectangle $[0,k]\times[0,l]$ to a walk in the square $[0,k+l]^2$.}
\end{figure}

\paragraph{Step 2: Squares.} We have bounded from below the number of
$n$-step self-avoiding walks `contained in a rectangle'. We now
extend this bound to the case of squares. There exist $k,l\leq n$ such
that the number of elements of $\Sigma_n$ with $(k,l)$ as an ending
point is larger than $e^{-2c\sqrt n}\mu^n/n^2$. By taking two
arbitrary walks of $\Sigma_n$ ending at $(k,l)$, one can construct a
$2n$-step self-avoiding walk with $\g_0=(0,0)$ and $\g_{2n}=(k+l,k+l)$
contained in $[0,k+l]^2$ by reflecting orthogonally to
$e^{i\pi/4}\mathbb R$ the first walk, and then concatenating the two, see
figure \ref{fig:rectangle_square}. We deduce that $a_{2n}\geq
\mu^{2n}e^{-4c\sqrt n}/n^4$. This shows the lemma for $n$ sufficiently
large, and one can increase $c$ if necessary to handle all even $n$.\qed

\begin{proof}[Proposition~\ref{proposition:loop}. ]
Squared walks with length $n$ were defined as walks between corners of
some $m\times m$ square, but $m$ was not fixed. Fix now $m$ to be such
that the number of such walks is maximized (and then it is at least $ a_n/n$
where $a_n$ is the total number of squared walks). It is interesting
to remark that finding the maximal $m$ as an explicit function of $n$,
even asymptotically, seems difficult, probably no easier than the
$\SLE_{8/3}$ conjecture. But we do not need to know its value.
From any quadruplet $(\g_1,\g_2,\g_3,\g_4)$ of such squared self-avoiding walks, one can construct a self-avoiding polygon of $P_{m}$ as follows (see figure \ref{fig:construction_loops}):
\begin{itemize}
\item translate $\g_1$ and $\g_3$ by $(m+1,0)$ and $(0,m+1)$ respectively,
\item rotate $\g_2$ and $\g_4$ by an angle $\pi/2$, and then translate them by $(m,0)$ and $(2m+1,m+1)$ respectively,
\item add the four edges $[(m,0),(m+1,0)]$, $[(2m+1,m),(2m+1,m+1)]$, $[(m,2m+1),(m+1,2m+1)]$ and $[(0,m),(0,m+1)]$.
\end{itemize}

Since each walk is contained in a square, one can easily check that we obtain a $(4n+4)$-step polygon in $P_{m}$. Using Lemma~\ref{lemma:number of rectangular SAWs}, we obtain
$$Z_{m}(x)\geq x^{4n+4} \left(\frac{a_n}{n}\right)^4\geq \left(\frac{x^{n+1}\mu^{n}e^{-c\sqrt n}}{n}\right)^4.$$
When $n$ goes to infinity, the right-hand side goes to infinity and the claim follows readily.\end{proof}

\section{Proof of the main results}

The strategy is the following. We first show that for some hole (namely it will be a connected  union of boxes of some size $m$), the probability that the self-avoiding walk gets close to it without intersecting it can be estimated in terms of $Z_m(x)$. This claim is the core of the argument, and is presented in Proposition~\ref{proposition:boxes}. Next, we show that choosing $m$ large enough (or equivalently $Z_m(x)$ large enough), the probability to avoid some connected union of $k$ boxes decays exponentially fast in $k$, thus implying Theorem~\ref{space filling improved}.

Let $m>0$. A {\em cardinal edge} of a (square) box $B$ of side length
$2m+1$ is an edge of the lattice in the middle of one of the sides of
$B$. For $m\in\mathbb N$, two boxes $B$ and $B'$ of side length $2m+1$
are said to be {\em adjacent} if they are disjoint and each has a
cardinal edge, $[xy]$ and $[zt]$ respectively, such that $x\sim z$,
$y\sim t$ (see figure~\ref{fig:FD_delta}). A family $F$ of boxes is
called {\em connected} if every two boxes can be connected by a path
of adjacent boxes in $F$. 

To simplify the picture, we will assume that
all our boxes have their lower left corner in $(2m+2)\d \Z^2$. When $\Omega_\delta$ is fixed, such boxes included in $\Omega_\delta$ are called $m$-{\em boxes} and the set of $m$-boxes is denoted by $\mathcal F(\Omega_\d,m)$.
\begin{figure}[t]
\begin{center}
\includegraphics[width=1.00\textwidth]{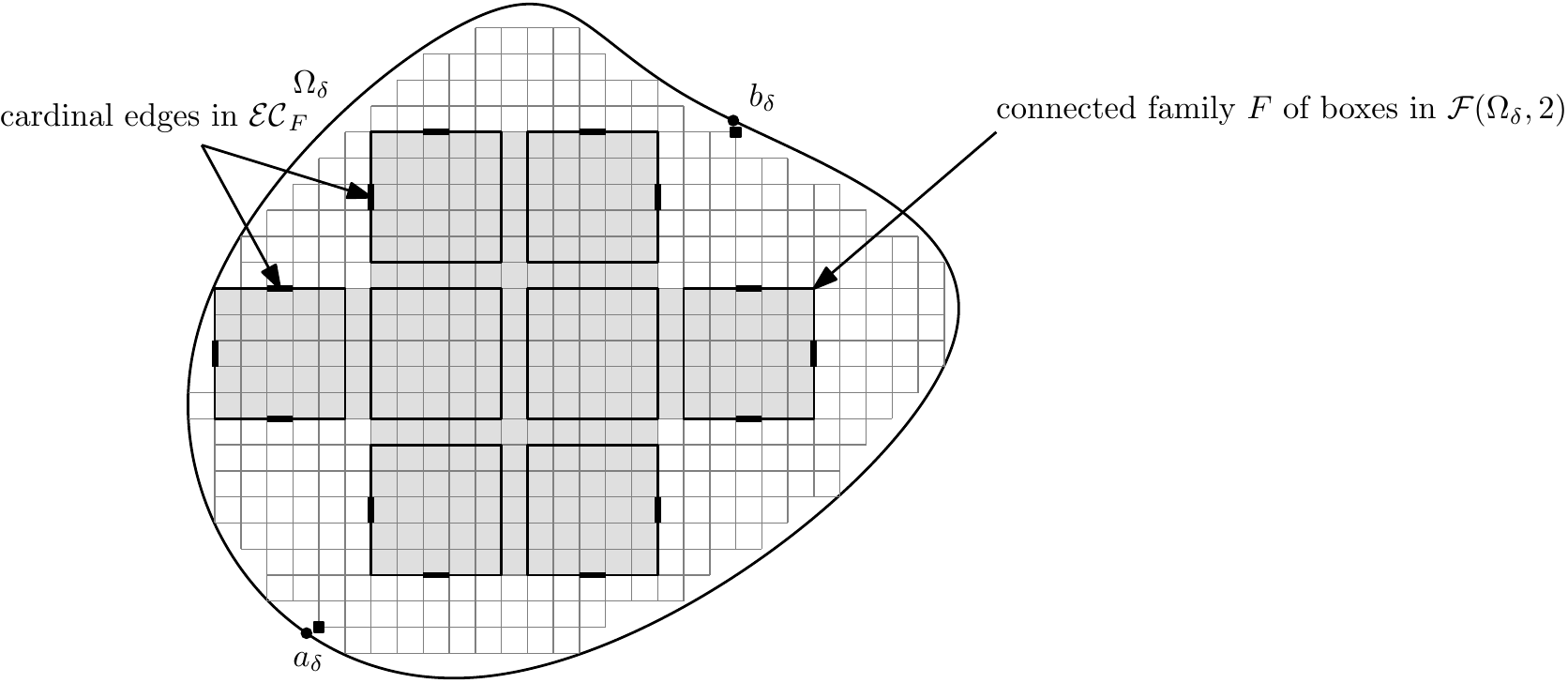}
\end{center}
\caption{\label{fig:FD_delta}A discrete domain with a connected
  component of adjacent boxes of size $5$ ($m=2$). Edges of $\mathcal E_F$ lie in the gray area.}
\end{figure}

Let us return to the issue of domain regularity discussed after
Theorem \ref{space filling improved}. With the definitions above we
can now explain that, in fact, our only requirement from the domain is
that the family of all boxes in $\mathcal F(\Omega_\d,m)$ is connected. Let us state
this formally. The set $\Gamma_\d^\xi$ is as
in Theorem \ref{space filling improved}.

\begin{theorem}\label{thm:connected} For every $x>1/\mu$, there exists
  $m=m(x)$ and $c(x)>0$ such that for every domain $\O$ and every $\d>0$ such that
$\mathcal F(\Omega_\d,m)$ is connected, one has, for every $a$
  and $b$ in the boundary of $\O$, and every $\lambda>0$,
\[
\P_{\O_\d,a_\d,b_\d,x}(\exists\mbox{ a component of
}\O_\d\setminus\Gamma_\d^{6m}\mbox{ of size }>\lambda )\le
\frac{C(x,\O)}{\d^2}e^{-c(x)\lambda}.
\]
\end{theorem}


\begin{proof}[Theorem \ref{space filling improved} given
    Theorem \ref{thm:connected}. ]
Here our domain is $\DD$. Clearly, the family of all boxes in $\DD_\d$
is connected (it is an interval in every row and every column), hence
Theorem \ref{thm:connected} applies. Taking $\lambda=C_1\log (\nicefrac
1\d)$ for $C_1$ sufficiently large gives the result.
\end{proof}
\begin{proof}[Theorem \ref{space filling expand} given Theorem
    \ref{thm:connected}. ]
Again, all we have to show is that the family of boxes in
$(\O+B(\xi\d))_\d$ is connected for $\xi$ sufficiently large. Taking
$\xi=6m$ we see that every box in $\O+B(6m\d)$ can be connected to a
box in $\O$, and any two boxes in $\O$ can be connected by taking a
path $\gamma$ in $\O$ between them (here is where we use that $\O$ is
connected) and checking that $\gamma+B(6m\d)$ contains a path of
connected boxes.
\end{proof}

Hence we need to prove Theorem~\ref{thm:connected}. Let $\d>0$. For $F\in \mathcal F(\Omega_\d,m)$, let $\mathcal V_F$ be the
set of vertices in boxes of $F$, and let $\mathcal E_F$ be the set of
edges with both end-points in $\mathcal V_F$. For two subsets $A$ and
$B$ of the vertices of $\O_\d$ define the {\em box distance}  $\bdist(A,B)$ between
them as the size of the smallest set of connected boxes containing one
box in $A$ and one box in $B$, minus 1. The boxes do not have to be
different, but then the distance is 0 --- if no such connected set
exists, then the distance is $\infty$. 

\begin{proposition}\label{proposition:boxes}
Let $(\Omega,a,b)$ be a domain with two points on the boundary. Fix
$\d>0$ and $m\in \N$ and assume $\mathcal F(\Omega_\d,m)$ is
connected. Then there exists $C(x,m)<\infty$ such that for every $F\in \mathcal F(\Omega_\d,m)$,
$$\P_{(\O_\d,a_\d,b_\d,x)}(\bdist(\g_\d,\mathcal V_F)= 1) \le C(x,m)\, Z_{m}(x)^{-|F|}.$$
\end{proposition}
\begin{proof}
For $F\in \mathcal F(\Omega_\d,m)$, let $\mathcal{EC}_F$ be the set of {\em
  external cardinal edges} of $F$ i.e.\ all cardinal edges in boxes of
$F$ which have neighbors outside of $F$. Let $S_F$ be the set of self-avoiding polygons included in $\mathcal E_F$ visiting all the edges in $\mathcal{EC}_F$. Let $Z_F(x)$ be the partition function of polygons in $S_F$. We have:

\begin{claim*}
For $F\in \mathcal F(\Omega_\d,m)$, $Z_F(x)\ge Z_m(x)^{|F|}$.
\end{claim*}
\begin{proof}[the claim.] We prove the result by induction on the cardinality of $F\in \mathcal F(\Omega_\d,\xi)$. If the cardinality of $F$ is 1, $Z_F(x)=Z_m(x)$ by definition. Consider $F_0\in \mathcal F(\Omega_\d,\xi)$  and assume the statement true for every $F\in \mathcal F(\Omega_\d,\xi)$ with $|F|<|F_0|$. There exists a box $B$ in $F_0$ such that $F_0\setminus \{B\}$ is still connected. Therefore, for every pair $(\g,\g')\in S_{\{B\}}\times S_{F_0\setminus\{B\}}$, one can associate a polygon in $S_F$ in a one-to-one fashion. Indeed, $B$ is adjacent to a box $B'\in F_0\setminus \{B\}$ so that one of the four cardinal edges (called $[ab]$) of $B$ is adjacent to a cardinal edge $[cd]$ of $B'$. Note that $[cd]$ belongs to $\mathcal {EC}_{F_0\setminus \{B\}}$. Then, by changing the edges $[cd]$ and $[ab]$ of $\g$ and $\g'$ into the edges $[ac]$ and $[bd]$, one obtains a polygon in $S_F$. Furthermore, the construction is one-to-one and we deduce
\begin{equation*}
Z_{F_0}(x)\ge Z_{F_0\setminus\{B\}}(x)Z_B(x) \ge 
Z_m^{|F_0\setminus\{B\}|}Z_m(x) = Z_m(x)^{|F_0|}.\tag*{\qedhere}
\end{equation*}
\end{proof}
Consider the set $\Theta_F$ of walks not intersecting $F$ yet reaching
to a neighboring box. Let $e$ be a cardinal edge of a box in $F$ which neighbours a box visited by $\gamma$.
For each $\g\in \Theta_F$, consider a self-avoiding polygon
$\ell=\ell(\gamma)$ ($\ell$ standing for ``link'') satisfying the following three properties:
\begin{itemize}
\item it contains $e$ and is included in $(\Omega_\d\setminus \mathcal E_F)\cup\{e\}$,
\item it intersects $\g$ either at just one edge, or at two adjacent edges only (to intersect means to intersect along edges),
\item it has length smaller than $10m+10$ (for simplicity, we will bound the length by $100m$).
\end{itemize}
One can easily check that such a polygon always exists, see figure \ref{fig:polygexist}.

\begin{figure}[t]
\begin{center}
\includegraphics[width=.50\textwidth]{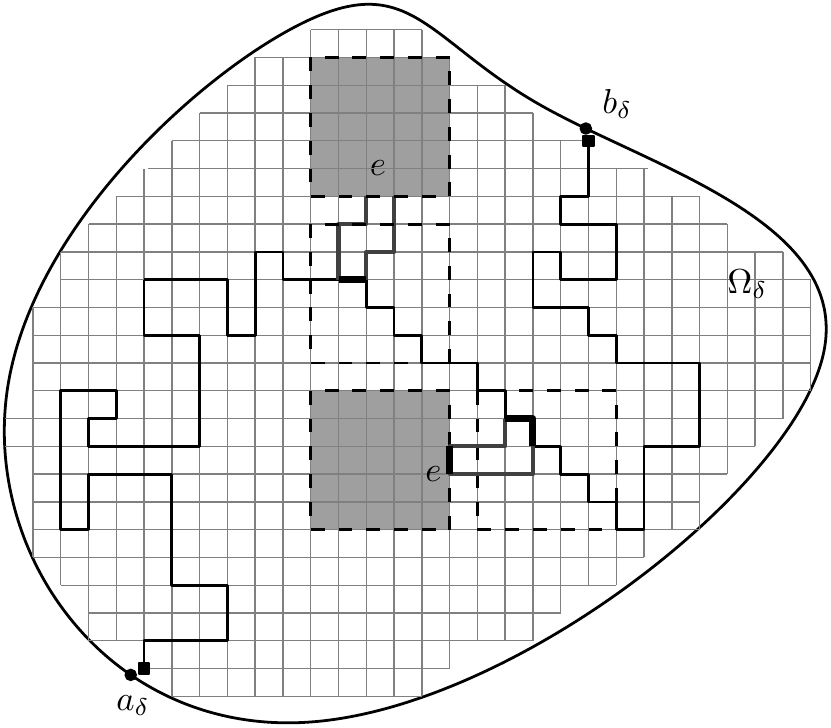}
\end{center}
\caption{\label{fig:polygexist} Two examples of the polygon $\ell$
  (the small N- and L-shaped loops). It overlaps the curve in one edge exactly, except in the second case, where we have no choice but overlapping the walk on two edges.}
\end{figure}

Now, consider the map $f$ that associates to
$(\g_1,\g_2)\in\Theta_F\times S_F$ the symmetric difference $\gamma=f(\g_1,\g_2)$ of
$\gamma_1$, $\ell(\gamma_1)$ and $\gamma_2$ (symmetric difference
here meaning as sets of edges). Note that the object
that we obtain is a walk from $a_\delta$ to $b_\delta$ in $\Omega_\delta$, which can be verified to be self-avoiding by
noting that each vertex has degree 0 or 2 and that the set is
connected. Further, its length is equal to $|\g_1|+|\ell(\g_1)|+|\g_2|-4$ or $|\g_1|+|\ell(\g_1)|+|\g_2|-6$ (this being due to the
fact that $\gamma_1$ and $\ell(\g_1)$ intersect at one or two
adjacent edges --- each intersection reduces the length by 2
edges). 
Now, given a path $\g$ there is a limited number of ways it may be
written as $f(\g_1,\g_2)$. Indeed, $e$ can be located by the only two
paths that exit $F$, and that gives $\g_2$. Given $e$, $\ell$ has only
a limited number of possibilities, say $4^{100m}$, and once one knows $\ell$
this gives $\g_1$. We can now write
\begin{align*}Z_{\Theta_F}(x)\cdot Z_F(x)&=\left(\sum_{\gamma_1\in \Theta_F}x^{|\gamma_1|}\right)\left(\sum_{\gamma_2\in S_F}x^{|\gamma_2|}\right)\\
&\le \max(1,x^{-100m})\sum_{\gamma_1\in \Theta_F,\g_2\in S_F}x^{|\gamma_1|+ |\ell(\gamma_1)|+|\gamma_2|}\\
&\le \max(1,x^{-100m})\max(x^4,x^6)\sum_{\gamma_1\in \Theta_F,\g_2\in S_F}x^{|f(\gamma_1,\gamma_2)|}\\
&\le 4^{100m} \max(x^6,x^{-100m+4})\sum_{\gamma\in f(\Theta_F\times S_F)}x^{|\gamma|}\\
&\le 4^{100m} \max(x^6,x^{-100m+4})Z_{(\O_\d,a_\d,b_\d)}(x),\end{align*}
where in the first inequality we used the fact that $\ell(\g_1)$ has length smaller than $100 m$, in the second the fact that $|f(\g_1,\g_2)|$ equals $|\g_1|+|\ell(
\g_1)|+|\g_2|-4$ or $|\g_1|+|\ell( \g_1)|+|\g_2|-6$, and in the third
the fact that $f$ is at most $100m^2$-to-one. Using the claim, the previous inequality implies
\begin{align*}
\P_{(\O_\d,a_\d,b_\d,x)}(\bdist(\g_\d,\mathcal V_F)=1) &=
\frac{Z_{\Theta_F}(x)}{Z_{(\O_\d,a_\d,b_\d)}(x)}\le\frac{C(x,m)}{Z_F(x)}\le\frac{C(x,m)}{Z_m(x)^{|F|}}.\qedhere
\end{align*}
\end{proof}


\begin{proof}[Theorem~\ref{thm:connected} in dimension 2.] Let
$x>1/\mu$ and let $(\Omega,a,b)$ be a domain with two points on the
boundary. Let $A_n$ be the number
of connected subsets of $\Z^2$ containing 0. It is well known that
$\varlimsup \sqrt[n]{A_n}$ is finite (see e.g.\ Theorem 4.20 in
\cite{GrimmettPerco}). Let therefore $\lambda=\lambda(2)$ satisfy
$A_n\le\lambda^n$ for all $n$. We now apply
Proposition~\ref{proposition:loop} and get some $m=m(x,2)$ such that
$Z_m(x)>2\lambda$. 

Let $\d>0$ and consider the event $\mathcal A(s)$ that there exists a
connected set $S$ of cardinality $s$ at distance larger than $6m$ of
$\g_\d$. Every box intersecting $S$ must be disjoint from $\g_\d$, so
there must exist a connected family of at least $s/(2m+1)^2$ boxes of
size $2m+1$ covering $S$ and not intersecting $\g_\d$. We may assume this family is maximal among
families covering $S$ and not intersecting $\gamma_\delta$. Since the
family of boxes is maximal, and because of the condition of the
theorem that the family of all boxes is connected, the box-distance between the union of boxes and $\g_\d$ is 1. Proposition~\ref{proposition:boxes} implies
$$\P_{(\O_\d,a_\d,b_\d,x)}[\mathcal A(s)]~\le~ \sum_{F\in\mathcal F(\Omega_\d,\xi):|F|
\geq s/(2m+1)^2}C(x,m)\left[Z_{m}(x)\right]^{-|F|}.$$
By the definition of $\lambda$, the number of families of connected
boxes of size $K$ in $\mathcal F(\Omega_\d,\xi)$ is bounded by
$(C(\Omega)/\delta^2) \lambda^{K}$ (since up to translation they are connected
subsets of a normalized square lattice), where $C(\Omega)=C(\Omega,x,m)$ depends on the area of $\Omega$, $x$ and $m$. Therefore, for $c>0$,
\begin{align*}\P_{(\O_\d,a_\d,b_\d,x)}\left[\mathcal A\left(s\right)\right]
&\le C(x,m)\frac{C(\Omega)}{\delta^2}\sum_{i\ge s /(2m+1)^2} \left(\frac{\lambda}{Z_{m}(x)}\right)^i\\
&\le \frac{C(x,m,\Omega)}{\delta^2}2^{-s/(2m+1)^2}\end{align*}
and the theorem follows.
\end{proof}

\begin{remark}\label{generalization}
Let us briefly describe what needs to be changed in higher
dimensions. The notion of cardinal edge must be extended: in the box
$[0,2m+1]^d$, cardinal edges for the face $[0,2m+1]^{d-1}\times\{0\}$
are all the edges joining vertices in $\{m,m+1\}^{d-1}\times\{0\}$
of the form $$\big[\,(\underbrace{m+1,\dots,m+1}_{i-1\text{ terms}},\underbrace{m,\dots,m}_{d-i\text{ terms}},0)\ ,\ (\underbrace{m+1,\dots,m+1}_{i\text{ terms}},\underbrace{m,\dots,m}_{d-i-1\text{ terms}},0)\,\big]$$
for $1\le i\le d-1$. See figure \ref{fig:cube} for an example in 3
dimensions. We only consider part of the edges joining vertices in $\{m,m+1\}^{d-1}\times\{0\}$ because all these edges should belong to a self-avoiding polygon. Similarly, cardinal edges can be defined for every face. It can be shown that the number of polygons included in some box $[0,2m+1]^d$ and visiting all the cardinal edges grows exponentially at the same rate as the number of self-avoiding walks. The proofs then apply {\em mutatis mutandis}. \end{remark}

\begin{figure}[t]
\begin{center}
\includegraphics{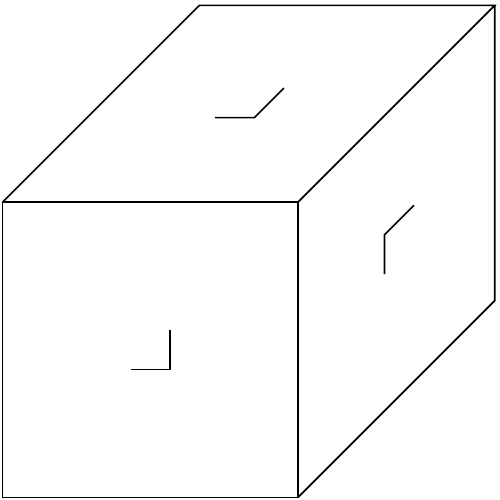}
\end{center}
\caption{\label{fig:cube}Cardinal edges in three dimensions.}
\end{figure}

\section{Questions}

The supercritical phase exhibits an interesting behavior. We know that the curve becomes space-filling, yet we have very little additional information. For instance, a natural question is to study the length of the curve. It is not difficult to show that the length is of order $1/\delta^2$, yet a sharper result would be interesting:
 \begin{problem}
 For $x>1/\mu$, show that there exists $\theta(x)>0$ such that for every $\varepsilon>0$ and every sufficiently regular domain $(\Omega,a,b)$,
 $$\mathbb P_{(\O_\d,a_\d,b_\d,x)}\big(\big|~|\g_\d|-\theta(x)\cdot|\Omega_\d|~\big|~>~\varepsilon~|\Omega_\d|\big)\longrightarrow 0\quad\text{when }\delta\rightarrow0.$$
 \end{problem}
The quantity $\theta(x)$ would thus be an averaged density of the walk. Note that the existence of $\theta(x)$ seems natural since the space-filling curve should look fairly similar in different portions of the space.

Another challenge is to try to say something nontrivial about the critical phase. Recently, the uniformly chosen self-avoiding walk on $\mathbb Z^d$ was proved to be sub-ballistic \cite{DH12}. A natural question would be to prove that it is {\em not} space-filling.
\begin{problem}
When $x=1/\mu$ and $(\O,a,b)$ is sufficiently regular, show that the sequence $(\g_\d)$ does not become space-filling.
\end{problem}

Finally, we recall the conjecture made in \cite{Sm2} concerning the
two-dimensional limit in the supercritical phase.

\begin{conjecture}[Smirnov]
Let $(\Omega,a,b)$  be a simply connected domain of $\mathbb C$ and consider approximations by the hexagonal lattice. The law of $(\g_\delta)$ converges to the chordal Schramm-L\"owner Evolution in $(\Omega,a,b)$\begin{itemize}
\item with parameter 8/3 if $x=1/\mu$,
\item with parameter 8 if $x>1/\mu$.
\end{itemize}
\end{conjecture}

\paragraph{Acknowledgements.} The first author was supported by the ANR grant BLAN06-3-134462,
the EU Marie-Curie RTN CODY, the ERC AG CONFRA, as well as by the Swiss
{FNS}. The second author was supported by the Israel Science Foundation.

\bibliographystyle{amsalpha}

\begin{thebibliography}{LSW04}

\bibitem[BDCGS12]{hugo}
R. Bauerschmidt, H. Duminil-Copin, J. Goodman, and G. Slade,
\emph{Lectures on self-avoiding-walks}, in Probability and Statistical Physics in Two and More Dimensions, Editors David Ellwood, Charles Newman, Vladas Sidoravicius, Wendelin Werner, published by CMI/AMS - Clay Mathematics Institute Proceedings (74 pages).

\bibitem[BDS]{BDS11}
D.~C. Brydges, A.~Dahlqvist, and G.~Slade, \emph{The strong interaction limit
  of continuous-time weakly self-avoiding walk}, available at:
  \href{http://arxiv.org/abs/1104.3731}{\nolinkurl{arXiv:1104.3731}}.

\bibitem[BIS09]{BrydgesImbrieSlade}
D.~C. Brydges, J.~Z. Imbrie, and G.~Slade, \emph{Functional integral
  representations for self-avoiding walk}, Probab. Surv. \textbf{6} (2009),
  34--61, available:
  \hrefnlu{http://www.i-journals.org/ps/viewarticle.php?id=152&layout=abstract%
}{i-journals.org}.

\bibitem[BS]{BS10}
D.~C. Brydges and G.~Slade, \emph{Renormalisation group analysis of weakly
  self-avoiding walk in dimensions four and higher}, lecture at ICM 2010,
  Hyderabad, available at:
  \href{http://arxiv.org/abs/1003.4484}{\nolinkurl{arXiv:1003.4484}}.

\bibitem[BS85]{BS85}
D.~C. Brydges and T.~Spencer, \emph{Self-avoiding walk in $5$ or more
  dimensions}, Commun. Math. Phys. \textbf{97} (1985), no.~1-2, 125--148,
  \hrefnlu{http://projecteuclid.org/DPubS?service=UI&version=1.0&verb=Display&%
handle=euclid.cmp/1103941982}{projecteuclid.org}.

\bibitem[DCH12]{DH12}
H.~Duminil-Copin and A.~Hammond, \emph{Self-avoiding walk is sub-ballistic}, available at: \hrefnlu{http://arxiv.org/abs/1205.0401}{arXiv:1205.0401}.

\bibitem[DCS12]{DS10}
H.~Duminil-Copin and S.~Smirnov, \emph{The connective constant of the
  honeycomb lattice equals $\sqrt{2+\sqrt 2}$}, Annals of Mathematics, \textbf{175(3)} (2012), 1653--1665. 

\bibitem[Flo53]{Flory}
P.~Flory, \emph{Principles of polymer chemistry}, Cornell University Press,
  1953.

\bibitem[Gri99]{GrimmettPerco}
G.~Grimmett, \emph{Percolation}, Springer Verlag, 1999.

\bibitem[HR17]{HR17}
G.~H. Hardy and S.~Ramanujan, \emph{The normal number of prime factors of a
  number n}, Quart. J. Pure Applied Math. \textbf{48} (1917), 76--92.

\bibitem[HS91]{HaraSlade1}
T.~Hara and G.~Slade, \emph{Critical behaviour of self-avoiding walk in five or
  more dimensions}, Bull. Amer. Math. Soc. (N.S.) \textbf{25} (1991), no.~2,
  417--423, available at:
  \hrefnlu{http://www.ams.org/journals/bull/1991-25-02/S0273-0979-1991-16085-4%
/home.html}{ams.org}.

\bibitem[HS92]{HaraSlade2}
\bysame, \emph{Self-avoiding walk in five or more dimensions. {I}. {T}he
  critical behaviour}, Commun. Math. Phys. \textbf{147} (1992), no.~1, 101--136,
  available:
  \hrefnlu{http://projecteuclid.org/DPubS?service=UI&version=1.0&verb=Display&%
handle=euclid.cmp/1104250528}{projecteuclid.org}.

\bibitem[HW62]{HammersleyWelsh}
J.~M. Hammersley and D.~J.~A. Welsh, \emph{Further results on the rate of
  convergence to the connective constant of the hypercubical lattice}, Quart.
  J. Math. Oxford Ser. (2) \textbf{13} (1962), 108--110, available at;
  \hrefnlu{http://qjmath.oxfordjournals.org/content/13/1/108.full.pdf+html}{ox%
fordjournals.org}.

\bibitem[Iof98]{Ioffe}
D.~Ioffe, \emph{Ornstein-Zernike behaviour and analyticity of shapes for
  self-avoiding walks on {${\bf Z}^d$}}, Markov Process. Related Fields
  \textbf{4} (1998), no.~3, 323--350.

\bibitem[LSW04]{LSW5}
G.~F. Lawler, O.~Schramm, and W.~Werner, \emph{On the scaling limit of planar
  self-avoiding walk}, Fractal geometry and applications: a jubilee of
  Beno\^\i t Mandelbrot, Part 2, Proc. Sympos. Pure Math., vol.~72, Amer.
  Math. Soc., Providence, RI, 2004, pp.~339--364. Available at:
  \hrefnlu{http://arxiv.org/abs/math/0204277}{arXiv:math/0204277},


\bibitem[MS93]{MadrasSlade}
N.~Madras and G.~Slade, \emph{The self-avoiding walk}, Probability and its
  Applications, Birkh\"auser Boston Inc., Boston, MA, 1993.

\bibitem[Smi06]{Sm2}
S.~Smirnov, \emph{Towards conformal invariance of 2D lattice models},
  International Congress of Mathematicians. Vol. {II}, Eur. Math. Soc.,
  Z\"urich, 2006, pp.~1421--1451. Available at:
  \hrefnlu{http://arxiv.org/abs/0708.0032}{arXiv:0708.0032}

\end{thebibliography}
\def\cprime{$'$}
\providecommand{\bysame}{\leavevmode\hbox to3em{\hrulefill}\thinspace}
\providecommand{\MR}{\relax\ifhmode\unskip\space\fi MR }
\providecommand{\MRhref}[2]{%
  \href{http://www.ams.org/mathscinet-getitem?mr=#1}{#2}
}
\providecommand{\href}[2]{#2}

\begin{flushright}
\footnotesize\obeylines
  \textsc{Universit\'e de Gen\`eve}
  \textsc{Gen\`eve, Switzerland}
  \textsc{E-mail:} \texttt{hugo.duminil@unige.ch}
 \medbreak
    \textsc{Weizmann Institute}
  \textsc{Rehovot, Israel}
  \textsc{E-mail:} \texttt{gady.kozma@weizmann.ac.il}
\medbreak
   \textsc{Ben Gurion University}
  \textsc{Beer Sheva, Israel}
  \textsc{E-mail:} \texttt{yadina@bgu.ac.il}
\end{flushright}

\end{document}